\documentclass{article}
\usepackage{graphicx}
\usepackage[leqno,intlimits]{amsmath}
\usepackage{amssymb}
\usepackage{amsthm}
\DeclareMathOperator{\des}{des}
\newtheorem{theorem}{Theorem}[section]
\newtheorem{prop}{Proposition}[section]
\begin{document}

\title{Comparing dealing methods with repeating cards}

\author{
	M\'arton Bal\'azs
	\thanks{Alfr\'ed R\'enyi Institute of Mathematics, Hungarian Academy of Sciences; \texttt{balazs@math.bme.hu}; research partially supported by the Hungarian Scientific Research Fund (OTKA) grants K100473, by T\'AMOP - 4.2.2.B-10/1-2010-0009, and by the Bolyai Scholarship of the Hungarian Academy of Sciences.}
	\and
	D\'avid Zolt\'an Szab\'o
	\thanks{E\"{o}tv\"{o}s Lor\'and University; \texttt{szdavid89@gmail.com}. Part of this work was done while the authors were affiliated with the Institute of Mathematics, Budapest University of Technology and Economics.}}

\maketitle

\begin{abstract}
In a recent work Conger and Howald derived asymptotic formulas for the randomness, after shuffling, of decks with repeating cards or all-distinct decks dealt into hands. In the latter case the deck does not need to be fully randomized: the order of cards received by a player is indifferent. They called these cases the ``fixed source'' and the ``fixed target'' case, respectively, and treated them separately. We build on their results and mix these two cases: we obtain asymptotic formulas for the randomness of a deck of repeating cards which is shuffled and then dealt into hands of players. We confirm that switching from ordered to cyclic dealing, or from cyclic to back-and-forth dealing improves randomness in a similar fashion than in the non-repeating ``fixed target'' case. Our formulas allow to improve even the back-and-forth dealing when the deck only contains two types of cards.

\end{abstract}

\section{Introduction}
In card-games a very important requirement is that, after shuffling, every hand dealt to players should have approximately the same probability. Therefore, the randomizing properties of the shuffling and dealing procedure is of essential interest.

In 1955 Gilbert and Shannon \cite{konyv} introduced the riffle shuffling as a mathematical model of card shuffling. In the 1980's Reeds \cite{cikk3} and Aldous \cite{cikk2} added the assumption that every possible cut/riffle combination is equally likely, and that has become known as the Gilbert-Shannon-Reeds or GSR model of card shuffling. First, the deck is cut into two packets of sizes $k$ and $n-k$ with probability $\frac{\binom{n}{k}}{2^n}$, $k=0\dots n$. After the cut the packets are combined together such that the cards of each packet maintain their relative order. It is assumed that each such interleaving is equally likely. As there are $\binom{n}{k}$ of them, any possible shuffling with a cut of size $k$ has probability $\frac{\binom{n}{k}}{2^n}\cdot \frac{1}{\binom{n}{k}}=\frac{1}{2^n}$. This probability does not depend on $k$, hence each pair of a cut and interleaving is equally likely.

In 1992, Bayer and Diaconis \cite{cikk5} generalized the riffle shuffling by introducing the $a$-shuffle to make the mixing problem easier. First the deck is cut into $a$ packets of sizes $p_1,\,p_2,\,\dots,\,p_a$, respectively, with probability $\frac{\binom{n}{p_1,p_2,\dots, p_a}}{a^n}$. Then the cards are interleaved such that the cards of each packet maintain their relative order, and each such interleaving is equally likely. It has been proved that making a random $a$-shuffle and then a random $b$-shuffle is equivalent to making a random $a\cdot b$ shuffle. In particular, this implies that a sequence of $i$ riffle shuffles is equivalent to a single $2^i$-shuffle.

Bayer and Diaconis used the variation distance:
\[
 \left|\left|\mathbb{P}_a-U\right|\right|:=\frac{1}{2}\sum_{\pi\in{S_n}}\left|\mathbb{P}_a(\pi)-U(\pi)\right|
\]
for their analysis, where $\mathbb{P}_a(\pi)$ is the probability of a particular permutation $\pi$ after an $a$-shuffle, $S_n$ is the symmetric group of degree $n$, $U$ represents the uniform distribution on permutations ($U$$(\pi)$=$\frac{1}{n!}$ for all $\pi\in{S_n}$), cards are distinct, and initially the deck is ordered (we start from the identity permutation). Bayer and Diaconis found an explicit formula for $\mathbb P_a$: \[\mathbb{P}_a(\pi)=\frac{1}{a^n}\binom{a+n-\des(\pi)-1}{n},\] where $n$ is the size of the deck and
\[
 \des(\pi):=\#\{i:\pi(i)>\pi(i+1)\}.
\]
In this paper we will consider permutations as a bijection from $\{1,2,...,k\}$ to itself, so if we apply $\pi$ to a sequence of objects, then object in position $i$ will move to position $\pi(i)$.  
This approach is illustrated via the next example: the permutation $\pi_1$=[43125] changes our initial ordering to 34215, as well as rearranging 25431 to 43521, and 53412 to 41352. It is easily checked that $\des(\pi_1)$=2.

An interesting generalization is when we allow some cards to have the same value. This makes our problem more complicated because decks (ordered sequences of cards) and transformations cannot be identified with permutations anymore. Indeed, there is a set of permutations for each pair of decks that transform the first deck into the second. Another novelty is that the initial configuration of a deck affects how fast the order of the cards approaches the uniform distribution. For a rearrangement $D'$ of $D$, let $S(D,D')$ be the set of permutations which transform $D$ into $D'$. The transition probability between $D$ and $D'$ is
\[
 \mathbb{P}_a(D \to D'):=\sum_{\pi\in{S(D,D')}}\mathbb{P}_a(\pi).
\]
Applying the above explicit formula, we arrive to
\[
 \mathbb{P}_a(D \to D')=\frac{1}{a^n}\sum_{d}b_d\binom{a+n-d-1}{n},
\]
where $b_d$ is the number of permutations in $S(D,D')$ with $d$ descents.

Conger and Viswanath \cite{cikk6} proved that the calculation of the transition probabilities is a \#P-complete problem. Most people believe that \#P-complete problems do not admit efficient solutions, so a possible way to examine this question is to approximate this probability. Conger and Howald \cite{cikk1} provided an approximation of the transition probabilities when $a$ is large. To describe their results we make some further definitions.
 
Let $a$ and $b$ be card values. We say that $D$ has an \emph{$a-b$ digraph at $i$}, if $D(i)=a$ and $D(i+1)=b$. We say that $D$ has an \emph{$a-b$ pair at ($i,j$)}, if $i<j$, $D(i)=a$, and $D(j)=b$. Let
\[
 \begin{aligned}
  W(D,a,b):&=\#\{a-b\text{ digraphs in }D\}-\#\{b-a\text{ digraphs in }D\},\\
  Z(D,a,b):&=\#\{a-b \text{ pairs in } D\}-\#\{b-a \text{ pairs in } D\}.
 \end{aligned}
\]
 As an example, the following deck, that consists of red ($R$) and black ($B$) cards, has 1 $R-B$ digraph, 2 $B-R$ digraphs, 12 $R-B$ pairs, 13 $B-R$ pairs:
\[
 D:=BRRRBBBBRR,
\]
and $W(D,B,R)=2-1=1$, $Z(D,B,R)=13-12=1$. Clearly, $W$ and $Z$ are antisymmetric in $a$ and $b$:
\[
 W(D,a,b)=-W(D,b,a),\qquad Z(D,a,b)=-Z(D,b,a).
\]

Conger and Howald \cite{cikk1} proved that
\begin{equation}\label{1}
 \mathbb{P}_a(D\to D')=\frac{1}{N}+c_1(D,D')a^{-1}+O(a^{-2}),
\end{equation}
where $N$ is the number of different reorderings of the deck represented by $D$, and 
\begin{equation}\label{2}
 c_1(D,D')=\frac{n}{2N}\sum_{a<b}\frac{W(D,a,b)Z(D',a,b)}{n_a n_b},
\end{equation}
with $n_a$ being the number of cards of value $a$, $n_b$ the number of cards of value $b$. They analyzed the behaviour of this formula in the case of repeated cards, where the complete order of the deck matters after the shuffling (``fixed source'' case). They also looked at shuffling and dealing into hands of all distinct cards where, on the other hand, only the cards dealt to players matter, but the order within a player is indifferent (``fixed target'' case). As a consequence it is shown that in the latter case with 52 distinct cards, switching from ordered dealing to cyclic dealing improves the randomness by a factor of 13, and switching from cyclic dealing to back and forth dealing again improves the randomness by a factor of 13.

Assaf, Diaconis and Soundarajan \cite{cikk7} also analyzed decks with repeated cards if only certain features are of interest, for instance, suits disregarded or only the colors of interest. For these features the number of shuffles drops in a significant rate.

In this paper we build on \eqref{1} and combine the above two cases: we derive a formula for the effectiveness of a dealing method when there are repeated cards in the deck, and only the hands dealt to players are of interest. Similarly to the all-distinct case we prove that, in this first-order approximation, switching from ordered dealing to cyclic dealing improves the randomness by a factor of $s$, $s$ being the number of cards each player receives. Switching from cyclic dealing to back and forth dealing improves the randomness by a factor of $s$ when $s$ is an odd number, while the coefficient $c_1(D,D')$ disappears for even $s$ values. Our formula becomes explicit enough so that for two types of cards and odd $s$ values we come up with a dealing method that is even better than back and forth.

The organization of the paper is as follows. In Section 2 we further introduce some notation and apply \eqref{1} to our case of repeated cards in the deck. In Section 3 we analyze the role of dealing methods with repeated cards, and arrive to a key formula in (\ref{6}) we can build on later. For simplicity, this is done with four players. In Section 4 we generalize the result to an arbitrary number of players, and compare the effectiveness of the ordered, the cyclic, and the back and forth dealing. In section 5 we briefly deal with the cases of non-ordered initial decks. Section 6 provides explicit computations when there are only two or three types of cards.

\section{The basics of our model}
We start with a deck of $4s$ cards. These cards can be repeated, their values (colours) are taken from the $k$-element set $\{P_1,\,P_2,\,\cdots,\,P_{k}\}$. The initial deck is ordered: the first $p_1$ cards are $P_1$ coloured, the next $p_2$ cards are $P_2$ coloured, $\dots$, and the last $p_k$ cards are $P_k$ coloured ($\sum_{i=1}^{k}p_i=4s$). An $a$-shuffle is performed on the deck, and then it is dealt to four players, called North(N), East(E), South(S) and West(W), respectively. The set $\Omega$ of hands consists of the vectors
\[
 \bar{p_i}=(p_{i,N},p_{i,E},p_{i,S},p_{i,W}),\qquad i=1\dots k,
\]
where $p_{i,N}$, $p_{i,E}$, $p_{i,S}$, $p_{i,W}$ is the number of $P_i$ coloured cards received by North, East, South, West, respectively. These numbers are non-negative integers, and satisfy
\[
 \begin{aligned}
  \sum_{i=1}^kp_{i,N}=\sum_{i=1}^kp_{i,E}=\sum_{i=1}^kp_{i,S}=\sum_{i=1}^kp_{i,W}&=s,\\
  p_{i,N}+p_{i,E}+p_{i,S}+p_{i,W}&=p_i,\qquad i=1\dots k.
 \end{aligned}
\]
Define $\Pi(\omega)$ as the stationary distribution, which is in fact uniform on $\Omega$:
\[
 \Pi(\omega)=\frac{s!^4}{(4s)!} \prod_{i=1}^k\frac{p_i !}{p_{i,N}!\cdot p_{i,E}!\cdot p_{i,S}!\cdot p_{i,W}!},\qquad\forall\omega\in\Omega.
\]
We use the variation distance as a level of randomness of the hands after an $a$-shuffling and dealing:
\[
 \begin{aligned}
  \left|\left|\mathbb{P}_a-\Pi\right|\right|:&=\frac{1}{2}\sum_{\omega\in{\Omega}}\left|\mathbb{P}_a(\omega)-\Pi(\omega)\right|\\
  &=\frac{1}{2}\sum_{\omega\in{\Omega}}\left|\sum_{D':(\bar{p_1},\bar{p_2},\dots,\bar{p_k})=\omega}\mathbb{P}_a(D\to D')-\Pi(\omega)\right|.
 \end{aligned}
\]
Here we suppose that $D$, the initial sequence of cards, is ordered, and $D'$ is some rearrangement of $D$. The inner sum is for all $D'$'s that give hand $\omega$ after the dealing. It is easy to see that this sum has
\begin{equation}\label{3}
 |D':(\bar{p_1},\bar{p_2},\cdots,\bar{p_k})=\omega|=s!^4 \prod_{i=1}^{k}\frac{1}{p_{i,N}!\cdot p_{i,E}!\cdot p_{i,S}!\cdot p_{i,W}!}
\end{equation}
terms. For computing $\mathbb{P}_a(D \to D')$ we use \eqref{1}, with $N=\binom{4s}{p_1,p_2,\dots,p_k}=\frac{(4s)!}{p_1!\cdot p_2!\cdots p_k!}$, and also \eqref{2}:
\begin{multline*}
\sum_{D':(\bar{p_1}, \bar{p_2}, \cdots, \bar{p_k})=\omega}\mathbb{P}_a(D \to D')\\
\begin{aligned}
&=\sum_{D':(\bar{p_1}, \bar{p_2}, \cdots, \bar{p_k})=\omega}\frac{1}{N}+a^{-1}\sum_{D':(\bar{p_1}, \bar{p_2}, \cdots, \bar{p_k})=\omega}c_1(D,D')\\
&\quad+\sum_{D':(\bar{p_1}, \bar{p_2}, \cdots, \bar{p_k})=\omega}O(a^{-2})\\
&=\Pi(\omega)+a^{-1}\!\!\sum_{D':(\bar{p_1}, \bar{p_2}, \cdots, \bar{p_k})=\omega}\frac{2s}{(4s)!}\Biggl(\prod_{j=1}^{k} p_j!\Biggl)\Biggl(\sum_{a<b}\frac{W(D,a,b)Z(D',a,b)}{n_a n_b}\Biggl)\\
&\quad+O(a^{-2})\\
&=\Pi(\omega)+a^{-1}\frac{2s}{(4s)!}\Biggl(\prod_{j=1}^{k}p_j!\Biggl)\Biggl(\sum_{i=1}^{k-1}\frac{\sum_{D':(\bar{p_1}, \bar{p_2}, \cdots, \bar{p_k})=\omega}Z(D',P_i,P_{i+1})} {p_i\cdot p_{i+1}}\Biggl)\\
&\quad+O(a^{-2}),
\end{aligned}
\end{multline*}
because $W(D,P_i,P_{i+1})=1$ ($i=1,\cdots,k-1$), and $W(D,A,B)=0$ if $B\neq (A+1)$ by virtue of the initial deck. Thus the variation distance becomes
\begin{multline}\label{5}
 \frac{1}{2}\sum_{\omega \in{\Omega}}\left|\mathbb{P}_a(\omega)-\Pi(\omega)\right|\\
 =\!a^{-1}\frac{s}{(4s)!}\Biggl(\prod_{j=1}^{k}p_j!\Biggl)\sum_{\omega \in{\Omega}}\left|\sum_{i=1}^{k-1}\frac{\sum_{D':(\bar{p_1}, \bar{p_2}, \cdots, \bar{p_k})=\omega}Z(D',P_i,P_{i+1})} {p_i\cdot p_{i+1}}\right|+O(a^{-2}).
\end{multline}

\section{Dealing methods}
Next we consider the role of the dealing methods. The dealer $a$-shuffles the deck and then deals it to the four players. We can describe the dealing method as a sequence of repeating letters $N$, $E$, $S$ and $W$, representing the order in which players receive the cards of the shuffled deck. The sequence that corresponds to the \emph{ordered dealing} is
\[
 NNN\dots NN\,EEE\dots EE\,SSS\dots SS\,WWW\dots WW.
\]
Here the first $s$ cards go to North, the next $s$ cards go to East, the next $s$ cards go to South and the last $s$ cards go to West. The next famous dealing method is the \emph{cyclic dealing} with
\[
 NESWNESWNESW\dots NESWNESW,
\]
where the top card goes to North, the second goes to East, the third goes to West, etc. The \emph{back and forth dealing} for even $s$ values is represented by
\[
 NESWWSENNESWWSEN\dots NESWWSEN,
\]
while for odd $s$ values we can write
\[
 NESWWSENNESWWSEN\dots NESW.
\]
Next we suppose that in the initial deck the first type is black (B), the second type is red (R). Let $b$ and $r$ denote the number of black and red cards in the deck, respectively, and let $p$ be the number of non-red and non-black cards. Let $N_p$, $E_p$, $S_p$, $W_p$ be the $p^\text{th}$ position, $p=1\dots s$, that goes to North, East, South, West, respectively in the dealing method. Let $p_N$, $p_E$, $p_S$, $p_W$ be the number of non-red and non-black cards that North, East, South, West has respectively after dealing. Let $b_N$, $b_E$, $b_S$, $b_W$ be the number of black cards that North, East, South, West has respectively after the dealing.

\begin{prop}\label{6}
\begin{multline*}
 \sum_{D':(\bar{b}, \bar{r}, \bar{p_3}, \bar{p_4}, \cdots, \bar{p_k})=\omega}Z(D',B,R)\\
 \begin{aligned}
  &=s!^4 \Bigl(\prod_{i=1}^{k}\frac{1}{p_{i,N}!\cdot p_{i,E}!\cdot p_{i,S}!\cdot p_{i,W}!}\Bigl)\cdot\Bigl((4s+1)b\\
  &\quad+\frac{b_N}{s}\Bigl(\frac{p_E\cdot Z(E,N)+p_S\cdot Z(S,N)+p_W\cdot Z(W,N)}{s}-2\sum_{p=1}^s N_p\Bigl)\\
  &\quad+\frac{b_E}{s}\Bigl(\frac{p_N\cdot Z(N,E)+p_S\cdot Z(S,E)+p_W\cdot Z(W,E)}{s}-2\sum_{p=1}^s E_p\Bigl)\\
  &\quad+\frac{b_S}{s}\Bigl(\frac{p_N\cdot Z(N,S)+p_E\cdot Z(E,S)+p_W\cdot Z(W,S)}{s}-2\sum_{p=1}^s S_p\Bigl)\\
  &\quad+\frac{b_W}{s}\Bigl(\frac{p_N\cdot Z(N,W)+p_E\cdot Z(E,W)+p_S\cdot Z(S,W)}{s}-2\sum_{p=1}^s W_p\Bigl)\Bigl),
 \end{aligned}
\end{multline*}
where
\begin{equation}
 \begin{aligned}
  Z(i,j)&=\#\{i-j \text{ pairs in the representing sequence of the dealing method}\}\\
	&-\#\{j-i \text{ pairs in the representing sequence of the dealing method}\}.
 \end{aligned}\label{eq:thmz}
\end{equation}
\end{prop}

\begin{proof}
First we will prove that for a particular permutation $D'$:
\begin{multline}
 Z(D',B,R)=\sum_{i=1}^{4s}(4s+1-2i\\
 \begin{aligned}
  &\quad+\text{(the number of non-red and non-black cards in $D'$}\\
  &\qquad\qquad\text{before the $i^{\text{th}}$ position)}\\
  &\quad-\text{(the number of non-red and non-black cards in $D'$}\\
  &\qquad\qquad\text{after the $i^{\text{th}}$ position)})\\
  &\quad\qquad\cdot\mathbf{1}\{\text{in the $i^{\text{th}}$ position there is a black card in permutation $D'$}\}.
 \end{aligned}\label{4}
\end{multline}
If we change in the position $i$ the value from red to black in a deck $D'$, then we can compute the change in the value $Z(D',B,R)$. Within the first $i-1$ cards, denote
\begin{itemize}
 \item by $A$ the number of black cards;
 \item by $C$ the number of red cards;
 \item by $G$ the number of non-red and non-black cards.
\end{itemize}
Suppose that the card in the position $i$ is red coloured. Furthermore, within the last $4s-i$ cards, denote
\begin{itemize}
 \item by $D$ the number of black cards;
 \item by $F$ the number of red cards;
 \item by $H$ the number of non-red and non-black cards.
\end{itemize}
Then we have $A+C+G=i-1$ and $D+F+H=4s-i$, and the card in the position $i$ stands in $\{B-R \text{ pairs}\}$ with those $A$ cards within the first $i-1$ cards which are black coloured; and stands in $\{R-B \text{ pairs}\}$ with those $D$ cards within the last $4s-i$ cards which are black coloured. If we change in the position $i$ the value from red to black we get that the change of the value $Z(D',B,R)$ is
\[
 \begin{aligned}
  &-A-C+D+F=-(i-1)+G+4s-i-H=4s+1-2i\\
  &+\text{(the number of non-red and non-black cards in $D'$ before the $i^{\text{th}}$ position)}\\
  &-\text{(the number of non-red and non-black cards in $D'$ after the $i^{\text{th}}$ position)}.
 \end{aligned}
\]
Based on this observation we now build up the value of $Z(D',B,R)$ recursively. We start with an all-red deck, in which the value of $Z(\cdot,B,R)$ is 0. First we flip from this deck all the non-black and non-red cards of $D'$. This does not change $Z(\cdot,B,R)$. Next we flip from red all the black cards of $D'$. Adding the changes in $Z(\cdot,B,R)$, we are lead exactly to \eqref{4}.

We introduce $U(i):=(\,$the number of non-red and non-black cards in $D'$ before the $i^{\text{th}}$ position$-$the number of non-red and non-black cards in $D'$ after the $i^{\text{th}}$ position$-2i$), and proceed with
\[
 \begin{aligned}
  &\sum_{D':(\bar{b}, \bar{r}, \bar{p_3}, \bar{p_4}, \cdots, \bar{p_k})=\omega} Z(D',B,R)=\sum_{D':(\bar{b}, \bar{r}, \bar{p_3}, \bar{p_4}, \cdots, \bar{p_k})=\omega}(4s+1)b\\
 +&\sum_{D':(\bar{b}, \bar{r}, \bar{p_3}, \bar{p_4}, \cdots, \bar{p_k})=\omega}\sum_{i=1}^{4s}(U(i))\\
  &\quad\cdot\mathbf{1}\{\text{in the $i^{\text{th}}$ position there is a black card in permutation $D'$}\}.
 \end{aligned}
\]
In order to compute $\sum_{D'}(U(i))$, we introduce an auxiliary uniform measure on the permutations. With the help of this measure we handle the sum as a conditional expectation of the random variable $U(i)$, a function of the permutation.
\begin{multline*}
 \begin{aligned}
  &\sum_{D':(\bar{b}, \bar{r}, \bar{p_3}, \bar{p_4}, \cdots, \bar{p_k})=\omega}(U(i))\\
  &\quad\cdot\mathbf{1}\{\text{in the $i^{\text{th}}$ position there is a black card in permutation $D'$}\}\\
  &=\Bigl(|D':(\bar{b}, \bar{r}, \bar{p_3}, \bar{p_4}, \cdots, \bar{p_k})=\omega, \text{in the $i^{\text{th}}$ position there is a black card}|\\
  &\quad\cdot\mathbf{E}(U(i)|\text{in the $i^{\text{th}}$ position there is a black card in permutation $D'$,}\\
 \end{aligned}\\
 (\bar{b}, \bar{r}, \bar{p_3}, \bar{p_4}, \cdots, \bar{p_k})=\omega)\Bigl),
\end{multline*}
for $i=1,\dots,4s$. The justification of this equality is that for each permutation $D'$ with $(\bar{b}, \bar{r}, \bar{p_3}, \bar{p_4}, \cdots, \bar{p_k})=\omega$ we sum up the value of $U(i)$ if position $i$ is black in permutation $D'$.
In 
\begin{multline*}
 \mathbf{E}(U(i)|\text{in position $i$ there is a black card in permutation }D',\\
 (\bar{b},\bar{r},\bar{p_3},\bar{p_4},\cdots,\bar{p_k})=\omega),
\end{multline*}
we also sum up these terms, but we divide each term by 
\[
 |D':(\bar{b},\bar{r},\bar{p_3},\bar{p_4},\cdots,\bar{p_k})=\omega,\text{ in position $i$ there is a black card}|.
\]
With this substitution we arrive to
\begin{multline}
 \sum_{D': (\bar{b}, \bar{r}, \bar{p_3}, \bar{p_4}, \cdots, \bar{p_k})=\omega}Z(D',B,R)\\
 \begin{aligned}
  &=|D':(\bar{b}, \bar{r}, \bar{p_3}, \bar{p_4}, \cdots, \bar{p_k})=\omega|(4s+1)b\\
  &+\!\Bigl(\sum_{i=1}^{4s}|D':(\bar{b}, \bar{r}, \bar{p_3}, \bar{p_4}, \cdots, \bar{p_k})=\omega, \text{in the $i^{\text{th}}$ position there is a black card}|\\
  &\quad\cdot\mathbf{E}(U(i)|\text{in the $i^{\text{th}}$ position there is a black card in permutation }D',\\
 \end{aligned}\\
 (\bar{b}, \bar{r}, \bar{p_3}, \bar{p_4}, \cdots, \bar{p_k})=\omega)\Bigl).\label{eq:mainsum}
\end{multline}
The dealing method determines which player receives the card in position $i$. Suppose it is player North and in the $i^{\text{th}}$ position there is a black card. Then
\[
 |D':(\bar{b},\bar{r},\bar{p_3},\bar{p_4},\cdots,\bar{p_k})=\omega|=b_N\cdot(s-1)!\cdot s!^3\prod_{j=1}^{k}\frac{1}{p_{j,N}!\cdot p_{j,E}!\cdot p_{j,S}!\cdot p_{j,W}!}.
\]
If position $i$ belongs to player East, South, or West, then the above formula holds with $b_N$ replaced by $b_E$, $b_S$, or $b_W$, respectively. Compare this to \eqref{3} to get
\begin{multline}
 \begin{aligned}
  &|D':(\bar{b}, \bar{r}, \bar{p_3}, \bar{p_4}, \cdots, \bar{p_k})=\omega, \text{in the $i^{\text{th}}$ position there is a black card}|\\
 =&|D':(\bar{b}, \bar{r}, \bar{p_3}, \bar{p_4}, \cdots, \bar{p_k})=\omega|
 \end{aligned}\\
 \begin{aligned}
 \cdot\Bigl(\frac{b_N}{s}&\mathbf{1}\{\text{position $i$ belongs to player North in the dealing method}\}\\
	   +\frac{b_E}{s}&\mathbf{1}\{\text{position $i$ belongs to player East in the dealing method}\}\\
	   +\frac{b_S}{s}&\mathbf{1}\{\text{position $i$ belongs to player South in the dealing method}\}\\
	   +\frac{b_W}{s}&\mathbf{1}\{\text{position $i$ belongs to player West in the dealing method}\} \Bigl).
 \end{aligned}\label{eq:iind}
\end{multline}
Next we turn to computing the conditional expectation.
\begin{equation}
 \begin{aligned}
  &\mathbf{E}(U(i)|\,\text{in the }i^{\text{th}}\text{ position there is a black card in permutation }D',\\
  &\quad\qquad(\bar{b}, \bar{r}, \bar{p_3}, \bar{p_4}, \cdots, \bar{p_k})=\omega)\\
  &=\sum_{j=1}^{i-1}\mathbf{P}(\text{position }j\text{ is non-red and non-black}\\
  &\qquad|\,\text{in the }i^{\text{th}}\text{ position there is a black card in permutation }D',\\
  &\quad\qquad(\bar{b},\bar{r},\bar{p_3},\bar{p_4},\cdots,\bar{p_k})=\omega)\\
  &-\sum_{j=i+1}^{4s}\mathbf{P}(\text{position }j\text{ is non-red and non-black}\\
  &\qquad|\,\text{in the }i^{\text{th}}\text{ position there is a black card in permutation }D',\\
  &\quad\qquad(\bar{b},\bar{r},\bar{p_3},\bar{p_4},\cdots,\bar{p_k})=\omega)-2i.
 \end{aligned}\label{eq:jsum}
\end{equation}
Suppose now that position $i$ belongs to player North in the dealing method. Then one of these probabilities can be computed in the following way:
\begin{multline}
 \begin{aligned}
  &\mathbf{P}(\text{position $j$ is non-red and non-black}\\
  &\quad|\,\text{in the $i^{\text{th}}$ position there is a black card in permutation $D'$},\\
  &\qquad(\bar{b}, \bar{r}, \bar{p_3}, \bar{p_4}, \cdots, \bar{p_k})=\omega)
 \end{aligned}\\
 \begin{aligned}
  =\frac{p_N}{s-1}&\mathbf{1}\{\text{position $j$ belongs to player North in the dealing method}\}\\
  +\frac{p_E}{s}&\mathbf{1}\{\text{position $j$ belongs to player East in the dealing method}\}\\
  +\frac{p_S}{s}&\mathbf{1}\{\text{position $j$ belongs to player South in the dealing method}\}\\
  +\frac{p_W}{s}&\mathbf{1}\{\text{position $j$ belongs to player West in the dealing method}\}.
 \end{aligned}\label{eq:jind}
\end{multline}
Now combine \eqref{eq:mainsum}, \eqref{eq:iind}, \eqref{eq:jsum} and \eqref{eq:jind} with the definition \eqref{eq:thmz} of $Z(i,j)$ (and notice that $Z(N,N)$ is trivially 0 for any dealing method) to conclude
\[
 \begin{aligned}
  &\Bigl(\sum_{i=1}^{s}|D': (\bar{b}, \bar{r}, \bar{p_3}, \bar{p_4}, \cdots, \bar{p_k})=\omega,\\
  &\qquad\text{in the $i^{\text{th}}$ position there is a black card in permutation $D'$}\\
  &\qquad\quad\text{and position $i$ belongs to player North}|\\
  &\quad\cdot\mathbf{E}(U(i)|\text{in the $i^{\text{th}}$ position there is a black card in permutation $D'$}\\
  &\qquad\quad\text{and position $i$ belongs to player North, }(\bar{b}, \bar{r}, \bar{p_3}, \bar{p_4}, \cdots, \bar{p_k})=\omega)\Bigl)\\
  &\quad=|D': (\bar{b}, \bar{r}, \bar{p_3}, \bar{p_4}, \cdots, \bar{p_k})=\omega|\\
  &\qquad\cdot\Bigl(\frac{b_N}{s}\Bigl(\frac{p_E\cdot Z(E,N)+p_S\cdot Z(S,N)+p_W\cdot Z(W,N)}{s}-2\sum_{p=1}^s N_p\Bigl)\Bigl).
 \end{aligned}
\]
Here we only considered those positions that belong to player North. Repeating the computation with positions that go to the other players we arrive to the statement of the proposotion (see also \eqref{3}).
\end{proof}

\section{The case of more players}
The generalization to the case of $\ell$ players and $\ell \cdot s$ cards of $k$ different colours is straightforward. The analogue of \eqref{5} now reads as:
\begin{multline*}
 \left|\left|\mathbb{P}_a-\Pi\right|\right|=a^{-1}\frac{s\cdot \Biggl(\prod_{j=1}^{k}p_j!\Biggl)}{(\ell s)!}\\
 \cdot\sum_{\omega \in{\Omega}}\left|\sum_{i=1}^{k-1}\frac{\sum_{D':(\bar{p_1}, \bar{p_2}, \cdots, \bar{p_k})=\omega}Z(D',P_i,P_{i+1})} {p_i\cdot p_{i+1}}\right|+O(a^{-2}),
\end{multline*}
where $p_1,p_2,\cdots,p_k$ are the number of cards coloured $P_1,P_2,\cdots,P_k$ in the deck. The main point is again the calculation of $\sum_{D':(\bar{p_1}, \bar{p_2}, \cdots, \bar{p_k})=\omega}Z(D',X, Y)$, where $X$ and $Y$ are two different types. Let $x_i$ be the number of $X$ coloured cards which player $i$ is dealt, and $x$ be the total number of cards of value $X$. Let $p_{o,j}$ be the number of $P_o$ coloured cards which player $j$ receives. (\ref{6}) generalizes to

\begin{prop}
\begin{multline*}
 \sum_{D':(\bar{p_1}, \bar{p_2}, \cdots, \bar{p_k})=\omega}Z(D',X, Y)\\
 \begin{aligned}
 &=|D':(\bar{p_1}, \bar{p_2}, \cdots, \bar{p_k})=\omega|\\
 &\quad\cdot\!\Bigl((\ell s+1)x\!+\!\frac{1}{s}\!\sum_{i=1}^{\ell}x_i\Bigl(\!\!\sum_{j=1, j\neq i}^{\ell}\!\!\frac{((\sum_{o=1}^k p_{o,j})-x_j-y_j) \cdot Z(j,i)}{s}-2\sum_{t=1}^s i_t\Bigl)\Bigl)\\
 &=|D':(\bar{p_1}, \bar{p_2}, \cdots, \bar{p_k})=\omega|\\
 &\quad\cdot\Bigl(\sum_{i=1}^{\ell}x_i \Bigl(\ell s+1-2\frac{\sum_{t=1}^{s}i_t}{s}+\!\!\sum_{j=1, j\neq i}^{\ell}\!\!\frac{((\sum_{o=1}^k p_{o,j})-x_j-y_j) \cdot Z(j,i)}{s^2}\Bigl) \Bigl),
 \end{aligned}
\end{multline*}
where $i_t$ is the $t^\text{th}$ position that goes to player $i$, $t=1\dots s$.
\end{prop}

We are now ready to compare the three famous dealing methods in terms of our first-order approximation \eqref{5}.
\begin{theorem}\label{tm:sfactor}
The coefficient of $a^{-1}$ is exactly $s$ times larger in the ordered dealing than in the cyclic dealing for every possible $k$ and $\ell$ values. If $s$ is even then the coefficient is 0 in the back and forth dealing. If $s$ is odd then the coefficient is exactly $s$ times smaller in the back and forth dealing than in the cyclic dealing for every possible $k$ and $\ell$ values.
\end{theorem}

\begin{proof}
In the ordered dealing, without loss of generality, suppose that player $j$ receives all his cards before player $i$ receives his first card. Then in the representing sequence there are $s^2$ $j-i$ pairs and 0 $i-j$ pairs, thus $Z(j,i)=s^2$.

In cyclic dealing, suppose that $j$ receives his first card before player $i$. Then the $p^\text{th}$ position that belongs to player $j$ stands in $j-i$ pair with $s-p+1$ positions that belong to player $i$, $p=1\dots s$. Similarly, the $p^\text{th}$ position that belongs to player $i$ stands in $i-j$ pair with $s-p$ positions that belong to player $j$, $p=1\dots s$. We conclude that $Z(j,i)=s$ in this case.

If $s$ is an even number then the representing sequence of the back and forth dealing is symmetric, hence $Z(j,i)=0$.

If $s$ is an odd number then in the representing sequence of the back and forth dealing let us call the first $\ell s-\ell$ positions the first group, the last $\ell$ positions the second group. The first group is symmetric, hence within the first group positions do not contribute to $Z(j,i)$. In the first group the $s-1$ positions that belong to player $j$ stand in $j-i$ pairs with the position which belongs to player $i$ in the second group, and the $s-1$ positions that belong to player $i$ stand in $i-j$ pairs with the position that belongs to player $j$ in the second group. Suppose again that player $j$ receives his first card before player $i$. Then in the second group the position that belongs to player $j$ stands in $j-i$ pair with the position that belongs to player $i$ in the second group. Thus we have $Z(j,i)=1$.

Summarizing, we have
\begin{itemize}
 \item Ordered dealing: $Z(j,i)=s^2 \cdot I_{j,i}$.
 \item Cyclic dealing: $Z(j,i)=s \cdot I_{j,i}$,
 \item Back and forth: $Z(j,i)=I_{j,i}$, if $s$ is an odd number.
 \item Back and forth: $Z(j,i)=0$, if $s$ is an even number,
\end{itemize}
where
\[
 I_{j,i}:=
 \begin{cases}
  1,&\text{player }j\text{ receives his first card before player }i\\
  -1,&\text{player }i\text{ receives his first card before player }j.
 \end{cases}
\]
This holds true for each pair $(i, j)$ of players, therefore the same holds for
\[
 \sum_{j=1, j\neq i}^{\ell}\frac{((\sum_{o=1}^k p_{o,j})-x_{j}-y_{j}) \cdot Z(j,i)}{s^2}:
\]
this is $s$ times larger for the ordered dealing than for the cyclic dealing, and this sum is $s$ times larger for the cyclic dealing than for the back and forth dealing with odd $s$ values, while the sum is zero for the back and forth dealing with even $s$ values.

Next we analyze the term $2\frac{\sum_{t=1}^{s}i_t}{s}$.

In the ordered dealing the positions $(i-1)s+1$, $(i-1)s+2,\cdots$, $is$ belong to the $i^{th}$ player, so we have to sum up these positions when we compute $\sum_{t=1}^{s}i_t$.
\[
 2\frac{\sum_{t=1}^{s}i_t}{s}=2\frac{((i-1)s+1+is)s}{2s}=2is-s+1.
\]

For cyclic dealing the positions $i$, $\ell+i, 2\ell+i, (s-1)\ell+i$ belong to the $i^{th}$ player, summing up these positions in $\sum_{t=1}^{s}i_t$ we have
\[
 2\frac{\sum_{t=1}^{s}i_t}{s}=2\frac{(2i+\ell(s-1))s}{2s}=2i+\ell(s-1).
\]

In the back and forth dealing  with even $s$, the positions $i$, $2\ell-(i-1), 2\ell+i, 4\ell-(i-1), 4\ell+i, 6\ell-(i-1), \cdots, \ell s-(i-1)$ belong to the $i^{th}$ player, summing up these positions for $\sum_{t=1}^{s}i_t$:
\[
2\frac{\sum_{t=1}^{s}i_t}{s}=2\frac{(\ell s+1)s}{2s}=(\ell s+1).
\]
If $s$ is an odd number, then positions $i$, $2\ell-(i-1), 2\ell+i, 4\ell-(i-1), 4\ell+i, 6\ell-(i-1), \cdots, \ell s-(\ell-i)$ belong to the $i^{th}$ player, and
\[
 2\frac{\sum_{t=1}^{s}i_t}{s}=\frac{2}{s}\Bigl(\frac{(1+\ell(s-1))(s-1)}{2}+\ell(s-1)+i\Bigl)=\frac{1}{s}(\ell s^{2}+s-1+2i-\ell).
\]
Therefore, we have
\[
 \ell s+1-2\frac{\sum_{t=1}^{s}i_t}{s}=s(\ell-2i+1)
\]
for the ordered dealing,
\[
 \ell s+1-2\frac{\sum_{t=1}^{s}i_t}{s}=(\ell-2i+1).
\]
for the cyclic dealing, and
\[
 \begin{aligned}
  \ell s+1-2\frac{\sum_{t=1}^{s}i_t}{s}&=0,\qquad&&s\text{ even,}\\
  \ell s+1-2\frac{\sum_{t=1}^{s}i_t}{s}&=\frac{1}{s}(\ell-2i+1),\qquad&&s\text{ odd}
 \end{aligned}
\]
in the back and forth dealing.

Thus we see that these terms also differ by factors of $s$ when comparing the ordered, cyclic, and back and forth dealing methods (odd $s$ values), while this term is also 0 for the back and forth dealing if $s$ is an even number. We have proved the claim for each term in the sum
\[
 \sum_{D':(\bar{p_1}, \bar{p_2}, \cdots, \bar{p_k})=\omega}Z(D',X, Y).
\]
which completes the argument.
\end{proof}
\section{The case of arbitrary initial deck}

Now, we suppose that the initial deck is arbitrary. In this case the variation distance is the following:
\begin{multline*}
 \frac{1}{2}\sum_{\omega \in{\Omega}}\left|\mathbb{P}_a(\omega)-\Pi(\omega)\right|\\
 \begin{aligned}
  &=a^{-1}\frac{s}{(\ell s)!}\Biggl(\prod_{j=1}^{k}p_j!\Biggl)\sum_{\omega \in{\Omega}}\left|\sum_{a<b}\frac{W(D,a,b)\sum_{D':(\bar{p_1}, \bar{p_2}, \cdots, \bar{p_k})=\omega}Z(D',a,b)} {n_a\cdot n_b}\right|\\
  &\quad+O(a^{-2}).
 \end{aligned}
\end{multline*}
The only term that depends on the dealing method is
\[
 \sum_{D':(\bar{p_1}, \bar{p_2}, \cdots, \bar{p_k})=\omega}Z(D',a,b).
\]
The proof of (\ref{tm:sfactor}) did not depend on the initial deck, hence that theorem extends to the case of an arbitrary initial deck.

\section{The case of two or three different types of cards}
The purpose of this section is to gain some quantitative insight on how the leading term of the variation distance behaves in the case of repetitive cards.
\subsection{Two different types of cards}
We now consider 52 cards in the deck, each either red or black, and four players. Let $b$ be the number of black cards in the deck. Using the computer and our formulas we are able to compute the coefficient of $a^{-1}$ for any possible value $b$. Applying (\ref{6}) the coefficient of $a^{-1}$ in \eqref{5} becomes
\[
 \begin{aligned}
&\ \frac{13}{b(52-b)\binom{52}{b}}\sum_{\omega \in \Omega}\Biggl|\binom{13}{b_N}\binom{13}{b_E}\binom{13}{b_S}\binom{13}{b_W}\Biggl(53b\\
&\ -\frac{2b_N}{13}\sum_{p=1}^{13}N_p-\frac{2b_E}{13}\sum_{p=1}^{13}E_p-\frac{2b_S}{13}\sum_{p=1}^{13}S_p-\frac{2b_W}{13}\sum_{p=1}^{13}W_p\Biggl)\Biggl|.
\end{aligned}
\]
The values of the last four sums are easily computed for a dealing method. The numerical values are 91; 260; 429; 598 for the ordered dealing, 325; 338; 351; 364 for the cyclic dealing, and 343; 344; 345; 346 for the back and forth dealing. We see that the values differ the least in the back and forth dealing and the most in the ordered dealing. In this sense the best of these dealing methods is the back and forth dealing, and the worst is the ordered dealing. We think that dealing method $A$ is better than dealing method $B$ if the following value is smaller in $A$ than in $B$:
\[
 |\sum_{p=1}^{13}N_p-344,5|+|\sum_{p=1}^{13}E_p-344,5|+|\sum_{p=1}^{13}S_p-344,5|+|\sum_{p=1}^{13}W_p-344,5|.
\]
Our conjecture is that the best dealing method is a dealing method in which two of the sums equal 344 and the other two equal 345. As an example, consider
\begin{multline}
 SNEWWSENNEWSWSENNESWWSENNESWWSEN\\
 NESWWSENNESWWSENNESW.\label{eq:dealg}
\end{multline}
Indeed, the graph \ref{s5} illustrates that this dealing method has approximately half the coefficient than that of the back and forth dealing for each value $b$.
\begin{figure}[ht]
 \begin{center}
   \includegraphics[scale=1.4]{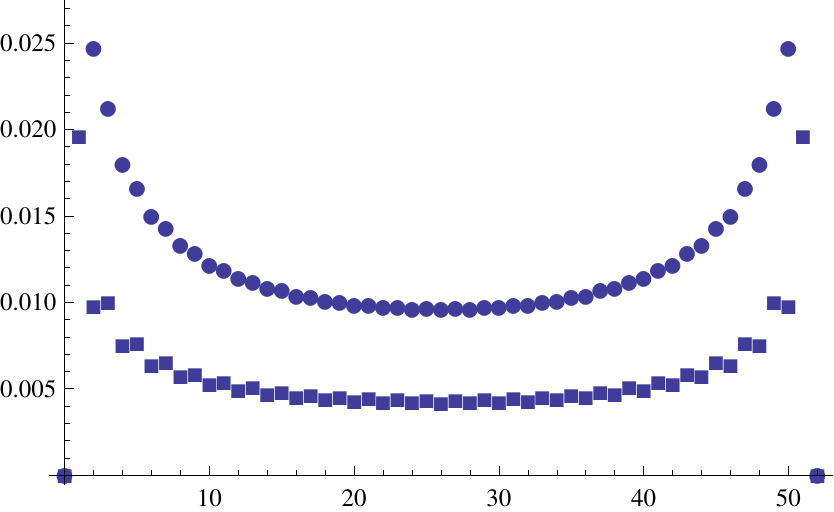}
   \caption{The coefficient of $a^{-1}$ for each possible value $b$ in the back and forth and in our conjectured best dealing method. The horizontal axis marks the value $b$ and the vertical axis marks the coefficient of $a^{-1}$. Squares plot our conjectured best dealing method and circles stand for back and forth. For better illustration we excluded the points $b=1$ and $51$ of the back and forth dealing.}
   \label{s5}
 \end{center}
\end{figure}
That is, for large $a$ we can save circa 1 riffle shuffle if we use this method instead of the back and forth dealing. We note that there are other dealing methods which have the same coefficient of $a^{-1}$ for each value $b$, but our conjecture is that there is no better one for two types of cards.
\subsection{Three different types of cards}
Next we suppose that there are three different types of cards in the deck: black, red and green. The number of cards is 52 and there are four players. Let $b$ be the number of black cards, $r$ be the number of red cards, $g$ be the number of green cards in the deck, $g=52-b-r$. Using the computer and our formula we are able to compute the coefficient of $a^{-1}$ for any possible $b, r$ values. Figure \ref{ize} shows the coefficient of $a^{-1}$ for each possible $b, r$ value in the back and forth dealing method.

\begin{figure}[ht]
 \begin{center}
   \includegraphics[scale=1]{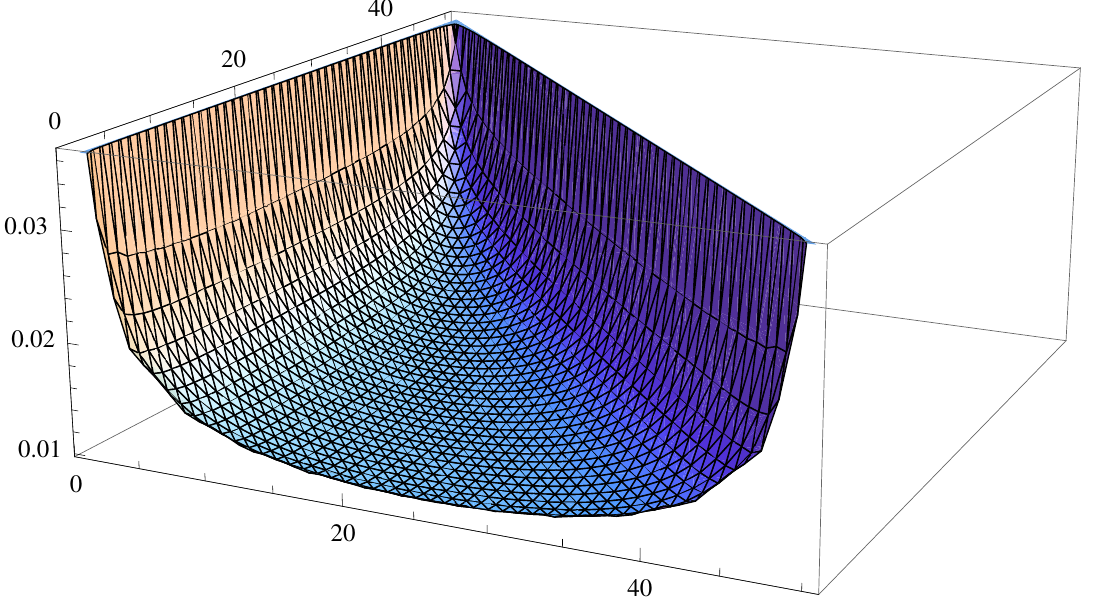}
   \caption{The coefficient of $a^{-1}$ for each possible value $b, r$ in the back and forth dealing method.}
  \label{ize}
 \end{center}
\end{figure}

A main question is whether the dealing method seen in \eqref{eq:dealg} is better than the back and forth dealing in this case. The answer is: not for every configuration. For $b=1, r=1$ the coefficient of $a^{-1}$ is $\frac{56}{1275}$ for the back and forth dealing and $\frac{76}{1275}$ for dealing method \eqref{eq:dealg}. Figure \ref{ize1} shows the coefficient of $a^{-1}$ for each possible $b, r$ values with \eqref{eq:dealg}. In most cases it has a smaller coefficient than the back and forth dealing, but there are some configurations when the back and forth dealing is better. This proves that the conjecture what we drew up for two different types is false for three types.

\begin{figure}[ht]
 \begin{center}
   \includegraphics[scale=1]{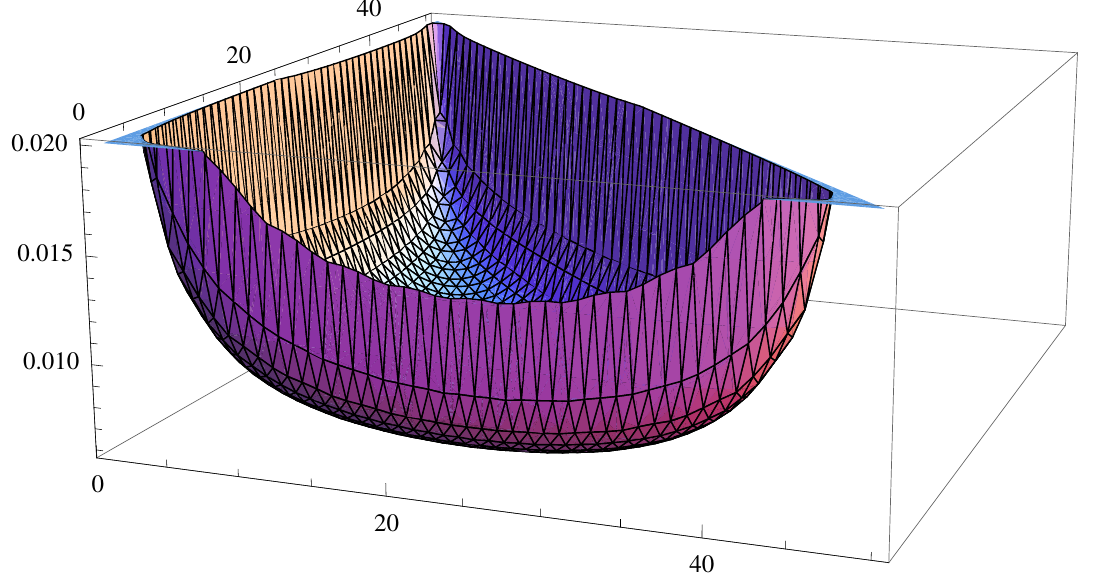}
   \caption{The coefficient of $a^{-1}$ for each possible value $b, r$ in dealing method \eqref{eq:dealg}.}
  \label{ize1}
 \end{center}
\end{figure}



\begin{thebibliography}{}
%
%
\bibitem{cikk2}
D.~Aldous.
\newblock Random walks on finite groups and rapidly mixing markov chains.
\newblock {\em Seminar on probability}, 986:243--297, 1983.

\bibitem{cikk5}
D.~Bayer and P.~Diaconis.
\newblock Trailing the dovetail shuffle to its lair.
\newblock {\em Ann. Appl. Probab.}, 2:294--313, 1992.

\bibitem{cikk1}
M.~Conger and J.~Howald.
\newblock A better way to deal the cards.
\newblock {\em American Mathematical Monthly}, 117:686--700, 2010.

\bibitem{cikk6}
M.~Conger and D.~Viswanath.
\newblock Riffle shuffles of decks with repeated cards.
\newblock {\em Annals of Probability}, 34:804--819, 2006.

\bibitem{cikk7}
S.~Assaf, P.~Diaconis and K.~Soundararajan
\newblock A rule of thumb for riffle shuffling.
\newblock {\em Ann. Appl. Probab.}, 21:843:875, 2011.

\bibitem{konyv}
E.~N. Gilbert.
\newblock {\em Theory of shuffling}.
\newblock Tech. Report MM-55-114-44, 1955.

\bibitem{cikk3}
J.~Reeds.
\newblock {\em Unpublished manuscript}.
\newblock 1981.
\end{thebibliography}
\end{document}